\documentclass[11pt,a4paper]{article}
\usepackage[letterpaper,top=2.7cm,bottom=2.7cm,left=2.5cm,right=2.5cm,marginparwidth=1.75cm]{geometry}
\DeclareMathAlphabet{\mathpzc}{OT1}{pzc}{m}{it}
\usepackage{amsmath,amsfonts,amsthm,amssymb}
\usepackage{amsmath,amssymb,color}
\usepackage{graphicx,fancyhdr}
\usepackage{mathrsfs,float}
\usepackage{multirow}
\usepackage{dsfont}
\usepackage{subcaption}
\usepackage{booktabs}
\usepackage{yfonts}
\usepackage{natbib}
\usepackage{caption}
\usepackage{geometry}
\usepackage{bbm}
\makeatletter
\usepackage{hyperref}

\numberwithin{equation}{section}
\newtheorem{thm}{Theorem}[section]
\newtheorem{rem}[thm]{Remark}

\newtheorem{lem}[thm]{Lemma}

\begin{document}

\title{Logarithmic stability for determining the damping coefficient for the time-fractional damped wave equation}
\author{
Kai Yu \thanks{Corresponding author, School of Mathematics and Statistics, Ningbo University, 818 Fenghua Road, Ningbo 315211, China. E-mail: yukaimailbox@163.com} \and
Zhiyuan Li \thanks{ School of Mathematics and Statistics, Ningbo University, 818 Fenghua Road, Ningbo 315211, China. } 
}

\date{}

\maketitle

\begin{abstract}

This paper investigates the inverse problem of determining the spatially
dependent damping coefficient in a time-fractional damped wave equation,
where the damping term is given by a Caputo derivative of order
\(\alpha\in(0,1)\). We first prove the well-posedness of the direct problem in
exponentially weighted Sobolev spaces. Then, by means of the Fourier--Laplace
transform in time, the nonstationary problem is reduced to a family of
stationary elliptic equations with complex frequencies. Based on complex
geometrical optics solutions and a suitable integral identity, we estimate the
Fourier transform of the difference of two damping coefficients in terms of the
difference of their Dirichlet-to-Neumann maps. Combining low-frequency estimates
with a high-frequency decay argument, we obtain a conditional logarithmic stability result. 



\end{abstract}

\section{Introduction}

Inverse coefficient problems for partial differential equations form an important
part of applied analysis and arise in many areas such as non-destructive testing,
medical imaging, and geophysical exploration \cite{Isa2006}. A central question
is whether spatially varying material parameters, such as conductivity, density,
potential, or damping coefficients, can be determined in a stable way from
boundary input-output measurements. This paper contributes to this topic by
studying a stability problem for a time-fractional damped wave equation. Such equations are used to describe media with memory and viscoelastic effects,
where the damping mechanism is nonlocal in time \cite{Mai2022}.

Let \(\mathbb R_+:=(0,\infty)\), and let \(\Omega\subset\mathbb R^n\) be a bounded
domain with smooth boundary \(\partial\Omega\). Set \(Q:=\Omega\times\mathbb R_+\) and \(\Sigma:=\partial\Omega\times\mathbb R_+.\)
We consider the following initial-boundary value problem:
\begin{equation}\label{eq1.1}
\left\{
\begin{array}{ll}
\partial_t^2 u(x,t)-\Delta u(x,t)
+p(x)\partial_t^\alpha u(x,t)+q(x)u(x,t)=0,
& (x,t)\in Q, \\[2mm]
u(x,t)=g(x,t), & (x,t)\in \Sigma, \\[2mm]
u(x,0)=0,\quad \partial_t u(x,0)=0, & x\in\Omega.
\end{array}
\right.
\end{equation}
Here \(\partial_t^\alpha\) denotes the Caputo derivative of order
\(\alpha\in(0,1)\), defined by
\[
\partial_t^\alpha u(x,t)
=
\frac{1}{\Gamma(1-\alpha)}
\int_0^t (t-\tau)^{-\alpha}\partial_\tau u(x,\tau)\,d\tau,
\qquad t>0,
\]
where \(\Gamma\) is the Gamma function. The coefficient \(p(x)\) represents the
damping coefficient, while \(q(x)\) is a potential term, which is assumed to be
known in this work. For a given boundary input \(g\), the corresponding
Dirichlet-to-Neumann (DtN) map is defined by
\[
\Lambda:g\mapsto \partial_\nu u|_{\partial\Omega\times\mathbb R_+}.
\]
The inverse problem considered in this paper is the stable determination of the
damping coefficient \(p(x)\) from the knowledge of \(\Lambda\).

The reconstruction of coefficients from boundary data has been extensively
studied. The Calderón problem, which asks whether the conductivity in an
elliptic equation can be determined from the DtN map, is a
prototype of such problems \cite{Cal1980}. The complex geometrical optics 
method introduced by Sylvester and Uhlmann \cite{SU1987} has become a fundamental
tool in this area and has led to many uniqueness and stability results; see, for
example, \cite{IUY2010,Nac1988,Nac1996}. For inverse coefficient problems for
classical wave equations, uniqueness of the potential from boundary
measurements was established in \cite{RS1988}. Further developments include the
boundary control method \cite{Bel1987} and Carleman estimate approaches
\cite{Kli1992}. Stability estimates for hyperbolic inverse problems have been
obtained from the DtN map \cite{Kian2014} and from partial
boundary observations \cite{KianSta2016}. The recovery of time-dependent
coefficients from a single boundary measurement \cite{Fei2023} and from partial
data on Riemannian manifolds \cite{KL2019} has also been investigated. A general
framework for stability in coefficient inverse problems for hyperbolic
equations, together with related numerical reconstruction schemes, was developed
in \cite{YLM2018}.

Inverse problems for damped wave equations have also attracted considerable
attention. For classical damped wave equations, Isakov \cite{Isa1991} proved
uniqueness for the simultaneous recovery of a damping coefficient and a
potential using boundary measurements. Recovery procedures based on the
DtN operator were proposed for the principal coefficient of a
damped wave equation in \cite{Rom2020}. Kian \cite{Kian2016} proved uniqueness
and stability for time-dependent damping coefficients and potentials from
partial boundary data, using Carleman estimates and geometric optics solutions.
Stability properties for related inverse problems from the initial-to-boundary
operator were studied in \cite{Amm2015,Amm2019}, where logarithmic and H\"older
type estimates were obtained under suitable assumptions.

The works mentioned above mainly concern local-in-time wave or damped wave
models. In contrast, wave equations with fractional attenuation involve
nonlocal time dynamics. Such models have been proposed in
\cite{CH2004,KMM2008,S1995} to describe power-law attenuation observed in
inhomogeneous media. Frequency-dependent attenuation laws are common in
applications \cite{S1994}, and time-nonlocal terms also arise naturally in
models of acoustic and viscoelastic wave propagation \cite{CH2003}. In
particular, fractional damping models are relevant in medical ultrasound and
other applications involving complex media. From the mathematical point of
view, the presence of memory kernels or fractional derivatives, such as
\(\partial_t^\alpha u\), changes the structure of the equation and makes the
corresponding inverse problems more delicate. Many methods for classical wave
equations do not apply directly to such nonlocal models.

There are relatively few results on inverse problems for wave equations with
time-nonlocal terms. Bukhgeim et al. \cite{BD1997} established stability
estimates for the recovery of a memory kernel from the DtN operator. Related uniqueness results were later extended to finite time
intervals \cite{Dya2003} and to partial boundary measurements \cite{BDU2001}.
The identification of memory kernels from other types of observations has also
been studied; see, for instance, \cite{Col2007,DS2015,JW2001}. More recently,
Kaltenbacher and Rundell \cite{Kal2022} investigated wave equations with
fractional damping and analyzed the interaction between the fractional order
and the damping profile. For further background on fractional models and
related inverse problems, we refer to the monograph \cite{Sel2016}.

In this paper, we investigate the inverse problem of determining the damping
coefficient \(p(x)\) in \eqref{eq1.1}. Our approach is based on a
Fourier--Laplace transform in time, which reduces the nonstationary problem to
a family of stationary elliptic equations with complex
frequency-dependent lower-order terms
\[
-p(x)(i\theta)^\alpha-q(x).
\]
We then construct complex geometrical optics solutions for the resulting
elliptic equations and derive a frequency-domain integral identity relating the
difference of two damping coefficients to the difference of the corresponding
DtN maps. By integrating the resulting estimates over a
suitable region in the frequency domain and by balancing the low-frequency and
high-frequency contributions, we prove a conditional logarithmic stability
estimate for the recovery of \(p(x)\). More precisely, for every
\(s\in(0,1/2)\), the stability exponent is
\[
\tau=\frac{2s}{n+4\alpha+2s}.
\]
This exponent can be chosen arbitrarily close to \(\frac{1}{n+4\alpha+1}\) by taking \(s\) sufficiently close to \(1/2\). The estimate also shows that, for
fixed \(n\) and \(s\), the stability deteriorates as the fractional order
\(\alpha\) increases.

The paper is organized as follows. Section \ref{sec2} introduces the functional
setting, states the well-posedness result for the direct problem, and formulates
the main stability theorem. Section \ref{sec3} is devoted to the proof of the
well-posedness of the direct problem. In Section \ref{sec4}, we construct the
complex geometrical optics solutions, derive the fundamental integral identity,
and prove the logarithmic stability estimate for the damping coefficient.
Finally, Section \ref{sec5} contains concluding remarks.


\section{Preliminaries and the main result}\label{sec2}

{
For the subsequent analysis, we first introduce the functional setting and several auxiliary results. Let \(L^p(\Omega)\), \(1\le p\le \infty\), denote the usual Lebesgue space on
\(\Omega\). The inner product in \(L^2(\Omega)\) is denoted by \((\cdot,\cdot)\),
and \(\|\cdot\|\) denotes the \(L^2(\Omega)\)-norm, unless otherwise specified.
For a nonnegative integer \(m\), we denote by \(W^{m,p}(\Omega)\) the Sobolev
space consisting of functions whose weak derivatives up to order \(m\) belong to
\(L^p(\Omega)\). We write \(H^m(\Omega):=W^{m,2}(\Omega).\)
The space \(H_0^m(\Omega)\) is defined as the closure of \(C_0^\infty(\Omega)\)
in \(H^m(\Omega)\). When the boundary is sufficiently smooth, this can be
equivalently characterized by the vanishing of boundary traces up to the appropriate order. For Banach spaces \(X\) and \(Y\), we denote by \(\mathcal B(X,Y)\) the space of bounded linear operators from \(X\) to \(Y\); see, e.g., \cite{AF2003}.

Let \(\gamma\ge0\) and let \(m\ge0\) be an integer. We define
\(H_\gamma^m(\Omega\times\mathbb R_+)\) as the space of all functions \(u\) such
that
\[
e^{-\gamma t}u\in H^m(\Omega\times\mathbb R_+).
\]
It is equipped with the norm
\[
\|u\|_{\gamma,m}^2
=
\int_{\Omega\times\mathbb R_+}
e^{-2\gamma t}
\sum_{|\beta|\le m}|D^\beta u(x,t)|^2\,dx\,dt .
\]
Here \(\beta\) denotes a multi-index in both the spatial and time variables.
Similarly, we define \(H_\gamma^m(\partial\Omega\times\mathbb R_+)\), and its
norm is denoted by \(\langle\cdot\rangle_{\gamma,m}\).

For the boundary data, we shall use the subspace
\[
\mathcal H_\gamma^1(\partial\Omega\times\mathbb R_+)
:=
\left\{
g\in H_\gamma^1(\partial\Omega\times\mathbb R_+):\ g(\cdot,0)=0\ \text{on }\partial\Omega
\right\},
\]
where the trace \(g(\cdot,0)\) is understood in the sense of the trace theorem.

We also recall the following well-posedness result for the wave equation in
exponentially weighted spaces. It follows from the standard well-posedness
theory for hyperbolic boundary value problems; see, for example, \cite{VG1980}.

\begin{lem}\label{lem wellp}
There exist constants \(\gamma_0>0\) and \(C>0\), depending only on \(\Omega\),
such that for every \(\gamma\ge\gamma_0\), \(g\in {\mathcal H}_\gamma^1(\partial\Omega\times\mathbb R_+)\) and \(f\in H_\gamma^0(\Omega\times\mathbb R_+),\)
the problem
\[
\left\{
\begin{array}{ll}
\partial_t^2 u(x,t)-\Delta u(x,t)=f(x,t),
& (x,t)\in Q, \\[2mm]
u(x,t)=g(x,t),
& (x,t)\in\Sigma, \\[2mm]
u(x,0)=0,\quad \partial_t u(x,0)=0,
& x\in\Omega,
\end{array}
\right.
\]
admits a unique solution \(u\in H_\gamma^1(\Omega\times\mathbb R_+).\) Moreover, the estimate
\[
\gamma\|u\|_{\gamma,1}^{2}
+
\left\langle\partial_{\nu}u\right\rangle_{\gamma,0}^{2}
\le
C\left(
\frac{1}{\gamma}\|f\|_{\gamma,0}^{2}
+
\langle g\rangle_{\gamma,1}^{2}
\right)
\]
holds. Here \(\partial_\nu\) denotes the outward normal derivative on \(\partial\Omega\).
\end{lem}

}
We now state the solvability theorem for the direct problem \eqref{eq1.1}:

\begin{thm}\label{thm wellp}
Let \(0<\alpha<1\), and assume that \(p,q\in L^\infty(\Omega)\). Then there
exists a constant \(\gamma_0>0\), depending only on \(\Omega\), \(\alpha\),
\(\|p\|_{L^\infty(\Omega)}\), and \(\|q\|_{L^\infty(\Omega)}\), such that for
every \(\gamma\ge\gamma_0\) and every
\(g\in{\mathcal H}_{\gamma}^{1}(\partial\Omega\times\mathbb R_+)\), problem
\eqref{eq1.1} admits a unique solution \(u\in H_{\gamma}^{1}(\Omega\times\mathbb R_+).\)
Moreover, \(\partial_\nu u\in H_{\gamma}^{0}(\partial\Omega\times\mathbb R_+)\),
and the estimate
\[
\gamma\|u\|_{\gamma,1}^2
+
\langle\partial_\nu u\rangle_{\gamma,0}^2
\le
C\langle g\rangle_{\gamma,1}^2
\]
holds, where \(C>0\) depends on \(\Omega\), \(\alpha\),
\(\|p\|_{L^\infty(\Omega)}\), and \(\|q\|_{L^\infty(\Omega)}\), but is independent
of \(g\) and \(\gamma\).
\end{thm}


By Theorem \ref{thm wellp}, for each admissible coefficient \(p\), we define the DtN map
\[
\Lambda: {\mathcal H}_{\gamma}^{1}(\partial \Omega \times \mathbb{R}_+)
\to H_{\gamma}^{0}(\partial \Omega \times \mathbb{R}_+),
\qquad
\Lambda g := \left.\partial_{\nu} u\right|_{\partial \Omega \times \mathbb{R}_+},
\]
where \(u\) denotes the solution to \eqref{eq1.1} with boundary input \(g\). The inverse problem considered here is to determine the unknown coefficient \(p(x)\) from the knowledge of \(\Lambda_p\). We prove a conditional stability estimate for this problem, showing that two admissible coefficients \(p_1\) and \(p_2\) must be close whenever the associated DtN maps \(\Lambda_{p_1}\) and \(\Lambda_{p_2}\) are close.

\begin{thm}[Conditional logarithmic stability]\label{thm1.3}
Let \( \Omega\subset\mathbb R^n\), \(n\ge 3\), be a bounded smooth domain, and let
\(0<\alpha<1\). Suppose that
\[
p_j,q\in C^2(\overline\Omega),\qquad
\|p_j\|_{C^2(\overline\Omega)}+\|q\|_{C^2(\overline\Omega)}\le K,
\qquad j=1,2.
\]
There exists \(\gamma_0>0\), depending only on \(\Omega\), \(\alpha\) and \(K\), such that for every
\(\gamma\ge \gamma_0\), the corresponding DtN maps \(\Lambda_j\) are well defined for problem \eqref{eq1.1} for \(j=1,2\). Then, for every \(s\in(0,1/2)\), there exist
a constant
\[
C=C(\Omega,n,\alpha,K,\gamma,s)>0
\]
and an exponent
\[
\tau=\frac{2s}{n+4\alpha+2s}>0
\]
such that
\[
\|p_1-p_2\|_{L^2(\Omega)}
\le
\omega\left(
\|\Lambda_1-\Lambda_2\|_{\mathcal B
\left(\mathcal H_\gamma^1(\partial\Omega\times\mathbb R_+),
H_\gamma^0(\partial\Omega\times\mathbb R_+)\right)}
\right),
\]
where the modulus of continuity satisfies
\[
\omega(\varepsilon)\sim C\bigl(\log(1/\varepsilon)\bigr)^{-\tau}
\quad \text{as } \varepsilon\to 0.
\]
\end{thm}

\begin{rem}
For some fixed \(n,s\), the logarithmic stability exponent \(\tau\) is decreasing with respect to
\(\alpha\). Thus, a larger fractional order in the nonlocal damping term leads
to a weaker stability estimate. Moreover, as \(s\uparrow 1/2\), the exponent \(\tau\) approaches \(1/(4\alpha+4)\) when \(n=3\).
\end{rem}


\section{Well-posedness of the direct problem}\label{sec3}

{
\begin{proof}[\textbf{Proof of Theorem \ref{thm wellp}}]
We first prove an estimate for the Caputo derivative in the weighted space.
Let \(\eta\in H^1_\gamma(\Omega\times\mathbb R_+)\) with
\(\eta(\cdot,0)=0\). Since \(0<\alpha<1\), we have
\[
\partial_t^\alpha \eta(x,t)
=
\frac{1}{\Gamma(1-\alpha)}
\int_0^t (t-\tau)^{-\alpha}\partial_\tau\eta(x,\tau)\,d\tau .
\]
Multiplying by \(e^{-\gamma t}\), we obtain
\[
e^{-\gamma t}\partial_t^\alpha\eta(x,t)
=
\frac{1}{\Gamma(1-\alpha)}
\int_0^t
e^{-\gamma(t-\tau)}(t-\tau)^{-\alpha}
e^{-\gamma\tau}\partial_\tau\eta(x,\tau)\,d\tau .
\]
Viewing the right-hand side as a convolution in time, we obtain, for almost every
\(x\in\Omega\), by Young's inequality,
\[
\|e^{-\gamma t}\partial_t^\alpha\eta(x,\cdot)\|_{L^2(\mathbb R_+)}
\le
\frac{1}{\Gamma(1-\alpha)}
\|t^{-\alpha}e^{-\gamma t}\|_{L^1(\mathbb R_+)}
\|e^{-\gamma t}\partial_t\eta(x,\cdot)\|_{L^2(\mathbb R_+)} .
\]
Since
\[
\int_0^\infty t^{-\alpha}e^{-\gamma t}\,dt
=
\gamma^{\alpha-1}\Gamma(1-\alpha),
\]
we get
\[
\|e^{-\gamma t}\partial_t^\alpha\eta(x,\cdot)\|_{L^2(\mathbb R_+)}
\le
\gamma^{\alpha-1}
\|e^{-\gamma t}\partial_t\eta(x,\cdot)\|_{L^2(\mathbb R_+)} .
\]
Squaring and integrating over \(\Omega\), we obtain
\begin{equation}\label{frac esti}
\|\partial_t^\alpha\eta\|_{\gamma,0}
\le
\gamma^{\alpha-1}\|\partial_t\eta\|_{\gamma,0}
\le
\gamma^{\alpha-1}\|\eta\|_{\gamma,1}.
\end{equation}

We now decompose the solution as \(u=w+v,\) where \(w\) solves the wave equation
\begin{equation}\label{wave homo}
\begin{cases}
\partial_t^2 w-\Delta w=0, & (x,t)\in Q,\\
w=g, & (x,t)\in \Sigma,\\
w(\cdot,0)=0,\,\,\partial_t w(\cdot,0)=0, &x\in\Omega.
\end{cases}
\end{equation}
By Lemma \ref{lem wellp}, there exist constants \(\gamma_0'>0\) and
\(C'>0\), depending only on \(\Omega\), such that, for every
\(\gamma\ge\gamma_0'\), problem \eqref{wave homo} has a unique solution
\(w\in H^1_\gamma(\Omega\times\mathbb R_+)\) satisfying
\begin{equation}\label{wave esti}
\gamma\|w\|_{\gamma,1}^2
+
\langle\partial_\nu w\rangle_{\gamma,0}^2
\le
C'\langle g\rangle_{\gamma,1}^2 .
\end{equation}
Then \(v=u-w\) must satisfy
\begin{equation}\label{wave f}
\begin{cases}
\partial_t^2 v-\Delta v
=
-p(x)\partial_t^\alpha(w+v)-q(x)(w+v)
=:f(x,t),
& (x,t)\in Q,\\
v=0, & (x,t)\in \Sigma,\\
v(\cdot,0)=0,\quad \partial_t v(\cdot,0)=0, &x\in\Omega.
\end{cases}
\end{equation}
Again by Lemma \ref{lem wellp}, there exist constants \(\gamma_0''>0\) and
\(C''>0\), depending only on \(\Omega\), such that, for every
\(\gamma\ge\gamma_0''\) and every \(f\in H^0_\gamma(\Omega\times\mathbb R_+)\),
the homogeneous boundary problem \eqref{wave f} has a unique solution
\(v\in H^1_\gamma(\Omega\times\mathbb R_+)\) satisfying
\begin{equation}\label{wave esti1}
\gamma\|v\|_{\gamma,1}^2
+
\langle\partial_\nu v\rangle_{\gamma,0}^2
\le
\frac{C''}{\gamma}\|f\|_{\gamma,0}^2 .
\end{equation}
Therefore, the solution operator
\[
E:H^0_\gamma(\Omega\times\mathbb R_+)
\longrightarrow
H^1_\gamma(\Omega\times\mathbb R_+),
\qquad
f\mapsto v,
\]
is well-defined and satisfies
\begin{equation}\label{E bound}
\|Ef\|_{\gamma,1}
\le
\frac{\sqrt{C''}}{\gamma}\|f\|_{\gamma,0}.
\end{equation}

We now solve \eqref{wave f} by a fixed-point argument. The equation for \(v\)
is equivalent to
\begin{equation}\label{fixed point}
v
=
E\left(
-p\partial_t^\alpha(w+v)-q(w+v)
\right).
\end{equation}
Set
\[
M:=\max\{\|p\|_{L^\infty(\Omega)},\|q\|_{L^\infty(\Omega)}\}.
\]
From \eqref{frac esti}, we obtain
\begin{equation}\label{f esti}
\begin{aligned}
\| f \|_{\gamma,0}
&\leq M \bigl( \| \partial_t^\alpha (w+v) \|_{\gamma,0}
+ \| w+v \|_{\gamma,0} \bigr) \\[2pt]
&\leq M \bigl( \gamma^{\alpha-1} \| w+v \|_{\gamma,1}
+ \| w+v \|_{\gamma,1} \bigr) \\[2pt]
&= M \bigl( \gamma^{\alpha-1} + 1 \bigr) \bigl(\| w\|_{\gamma,1}+\|v\|_{\gamma,1}\bigr).
\end{aligned}    
\end{equation}

Define
\[
\Phi(v)
:=
E\left(
-p\partial_t^\alpha(w+v)-q(w+v)
\right),
\qquad
v\in H^1_\gamma(\Omega\times\mathbb R_+).
\]
Then \(\Phi\) maps \(H^1_\gamma(\Omega\times\mathbb R_+)\) into itself. By
\eqref{E bound} and \eqref{f esti},
\begin{equation}\label{mapping}
\| \Phi(v) \|_{\gamma,1}
\leq \frac{\sqrt{C''}}{{\gamma}} \| f \|_{\gamma,0} \leq \frac{\sqrt{C''}\,M\bigl( \gamma^{\alpha-1}+1 \bigr)}{{\gamma}}
\; \bigl( \| w \|_{\gamma,1} + \| v \|_{\gamma,1} \bigr).   
\end{equation}
Moreover, for any \(v_1,v_2\in H^1_\gamma(\Omega\times\mathbb R_+)\),
\[
\begin{aligned}
\|\Phi(v_1)-\Phi(v_2)\|_{\gamma,1}
&\le
\frac{\sqrt{C''}}{\gamma}
\left\|
p\partial_t^\alpha(v_1-v_2)+q(v_1-v_2)
\right\|_{\gamma,0}  \\
&\le
\frac{\sqrt{C''}M(\gamma^{\alpha-1}+1)}{\gamma}
\|v_1-v_2\|_{\gamma,1}.
\end{aligned}
\]
Set
\[
A_\gamma
:=
\frac{\sqrt{C''}M(\gamma^{\alpha-1}+1)}{\gamma}.
\]
Since \(0<\alpha<1\), we have \(A_\gamma\to0\) as \(\gamma\to\infty\).
Choose \(\gamma_0\ge \max\{\gamma_0',\gamma_0'',1\}\)
large enough such that
\[
A_\gamma\le \frac12,
\qquad
\gamma\ge\gamma_0.
\]
Then \(\Phi\) is a contraction on the Banach space
\(H^1_\gamma(\Omega\times\mathbb R_+)\). Hence there exists a unique fixed point
\(v\in H^1_\gamma(\Omega\times\mathbb R_+)\) solving \eqref{fixed point}.
Consequently, \(u=w+v\) is a solution of the original problem.

It remains to derive the a priori estimate. Since \(v=\Phi(v)\), by
\eqref{mapping} we have
\[
\|v\|_{\gamma,1}
\le
A_\gamma\left(\|w\|_{\gamma,1}+\|v\|_{\gamma,1}\right).
\]
Using \(A_\gamma\le1/2\), we obtain
\begin{equation}\label{v esti}
\|v\|_{\gamma,1}
\le
\frac{A_\gamma}{1-A_\gamma}\|w\|_{\gamma,1}
\le
2A_\gamma\|w\|_{\gamma,1}.
\end{equation}
Next, by \eqref{wave esti1} and \eqref{f esti},
\[
\begin{aligned}
\gamma\|v\|_{\gamma,1}^2
+
\langle\partial_\nu v\rangle_{\gamma,0}^2
&\le
\frac{C''}{\gamma}
M^2(\gamma^{\alpha-1}+1)^2
\left(\|w\|_{\gamma,1}+\|v\|_{\gamma,1}\right)^2 .
\end{aligned}
\]
Using \eqref{v esti} and \(A_\gamma\le1/2\), we have
\[
\|w\|_{\gamma,1}+\|v\|_{\gamma,1}
\le
(1+2A_\gamma)\|w\|_{\gamma,1}
\le
2\|w\|_{\gamma,1}.
\]
Therefore, from \eqref{wave esti}, we get
\begin{equation}\label{v energy estimate}
\gamma\|v\|_{\gamma,1}^2
+
\langle\partial_\nu v\rangle_{\gamma,0}^2
\le
\frac{4C'C''M^2(\gamma^{\alpha-1}+1)^2}{\gamma^2}
\langle g\rangle_{\gamma,1}^2\le C\langle g\rangle_{\gamma,1}^2,
\end{equation}
where \(C>0\) is independent of \(g\). Here we use the fact that the factor \(\frac{(\gamma^{\alpha-1}+1)^2}{\gamma^2}\) is bounded uniformly for \(\gamma\ge\gamma_0\). 

Finally, since \(u=w+v\), we have
\[
\|u\|_{\gamma,1}^2
\le
2\|w\|_{\gamma,1}^2+2\|v\|_{\gamma,1}^2,
\]
and
\[
\langle\partial_\nu u\rangle_{\gamma,0}^2
\le
2\langle\partial_\nu w\rangle_{\gamma,0}^2
+
2\langle\partial_\nu v\rangle_{\gamma,0}^2.
\]
Combining \eqref{wave esti} and \eqref{v energy estimate}, we obtain
\[
\begin{aligned}
\gamma\|u\|_{\gamma,1}^2
+
\langle\partial_\nu u\rangle_{\gamma,0}^2
\le 2\left(
\gamma\|w\|_{\gamma,1}^2
+
\langle\partial_\nu w\rangle_{\gamma,0}^2
\right) +
2\left(
\gamma\|v\|_{\gamma,1}^2
+
\langle\partial_\nu v\rangle_{\gamma,0}^2
\right) \le C\langle g\rangle_{\gamma,1}^2,
\end{aligned}
\]
where \(C>0\) depends only on \(\Omega\), \(\alpha\),
\(\|p\|_{L^\infty(\Omega)}\), and \(\|q\|_{L^\infty(\Omega)}\), but is independent
of \(g\) and \(\gamma\ge\gamma_0\). Moreover, uniqueness of the solution follows from the contraction principle applied to the difference of two solutions. This completes the proof.
\end{proof}

}

\section{Stability estimate of the inverse problem}\label{sec4}

\begin{proof}[\textbf{Proof of Theorem \ref{thm1.3}}]

We now give the proof of the conditional logarithmic stability estimate.
The overall strategy is as follows. First, we apply the Fourier--Laplace
transform in time, which reduces the non‑stationary problem to a family of
stationary elliptic equations parameterized by the complex frequency
\(\theta=\sigma-i\gamma\). In the frequency domain we construct complex
geometrical optics (CGO) solutions and derive a key integral identity relating
\(P=p_1-p_2\) to the difference of the DtN maps
\(\Lambda_1^\theta-\Lambda_2^\theta\). By estimating the resulting terms and
integrating over a suitable low‑frequency region, we obtain an upper bound for
the Fourier transform \(\widetilde P\). Combining this with a high‑frequency
decay estimate, which follows from the \(H^s\)-regularity of the zero extension
of \(P\) for \(s<1/2\), and then choosing the free parameters appropriately, we
balance the low‑ and high‑frequency contributions. This yields a logarithmic
bound for \(\widetilde P\) in \(L^2(\mathbb R^n)\), and the desired stability
estimate follows from Plancherel's theorem.

In what follows, \(C\) denotes a generic positive constant that may change from
line to line and depends only on the a priori data \(\Omega, n, \alpha, K,
\gamma\) and, where applicable, on the Sobolev index \(s\).

\medskip

The construction of CGO solutions relies on the following lemma, which is a
generalization of a classical result by Sylvester and Uhlmann \cite{SU1987}
(see also \cite{Isak1991}).  In the statement below, which was proved in
\cite{BD1997}, the dot denotes the ``inner product without complex conjugation'';
i.e., \(x\cdot y = x_1y_1+\cdots+x_ny_n\).

\begin{lem}\label{lem ell eq}
Consider the equation
\begin{equation}\label{eq spe solu}
(\Delta+\theta^2)v(x)+a(x)v(x)=0
\quad\text{in }\Omega .
\end{equation}
Assume that \(a\in C^\ell(\overline\Omega)\) for some integer \(\ell\ge0\).
Then there exist constants \(C_1,C_2,C_3,C_4>0\) such that, for every
\(\zeta\in\mathbb C^n\) satisfying
\[
\zeta\cdot\zeta+\theta^2=0,
\qquad
|\operatorname{Im}\zeta|\ge\sqrt2,
\qquad
|\zeta|\ge C_1\|a\|_{C^\ell(\overline\Omega)},
\]
equation \eqref{eq spe solu} admits a solution of the form
\[
v(x,\zeta)=e^{\zeta\cdot x}\bigl(1+w(x,\zeta)\bigr).
\]
Moreover,
\[
\|w(\cdot,\zeta)\|_{H^\ell(\Omega)}
\le
\frac{C_2}{|\zeta|}
\|a\|_{C^\ell(\overline\Omega)},\quad \|v(\cdot,\zeta)\|_{H^\ell(\Omega)}
\le
C_3 e^{C_4|\zeta|}.
\]
Here \(C_4\) depends only on \(\Omega\), while \(C_1,C_2\), and \(C_3\) depend on
\(\Omega\) and \(\ell\).
\end{lem}

With this lemma at hand, we proceed to the proof.

\medskip
\noindent
\textbf{Step 1.  Fourier--Laplace reduction.} 

Let \(h(t)\) be a causal function defined on \(\mathbb R_+=(0,\infty)\), and denote by the same symbol its zero extension to \(\mathbb R\). For fixed \(\gamma>0\), its Fourier--Laplace transform is
\[
\widehat h(\theta) = \frac1{\sqrt{2\pi}} \int_0^\infty e^{-i\theta t} h(t)\,dt,
\qquad \theta = \sigma - i\gamma,\; \sigma\in\mathbb R .
\]
Equivalently,
\[
\widehat h(\sigma-i\gamma) = \mathcal F_t\bigl(e^{-\gamma t}h(t)\bigr)(\sigma).
\]
By Plancherel's theorem, for any Hilbert space \(X\), we have
\begin{equation}\label{eq:weighted-plancherel-main}
\int_{\mathbb R} \|\widehat h(\sigma-i\gamma)\|_X^2\,d\sigma
= \int_0^\infty e^{-2\gamma t}\|h(t)\|_X^2\,dt .
\end{equation}

Let \(g\in \mathcal H_\gamma^1(\partial\Omega\times\mathbb R_+)\) and let \(u_j\) \((j=1,2)\) be the solution of \eqref{eq1.1} corresponding to the coefficient \(p_j\). Applying the Fourier--Laplace transform in time to
\[
\partial_t^2 u_j - \Delta u_j + p_j(x)\partial_t^\alpha u_j + q(x)u_j = 0,
\]
and using the vanishing initial data \(u_j(\cdot,0)=\partial_t u_j(\cdot,0)=0\), we obtain
\[
\widehat{\partial_t^2 u_j} = (i\theta)^2 \widehat u_j = -\theta^2 \widehat u_j .
\]
For the fractional derivative, the definition of the Caputo derivative together with the initial conditions yields
\[
\widehat{\partial_t^\alpha u_j}(\theta) = (i\theta)^\alpha \widehat u_j(\theta),
\]
where the principal branch is taken; this is well defined because \(\operatorname{Re}(i\theta)=\gamma>0\).

Hence, for almost every \(\sigma\in\mathbb R\), the Fourier--Laplace transform
\(\widehat u_j(\cdot,\theta)\), with \(\theta=\sigma-i\gamma\), satisfies the
stationary boundary value problem
\begin{equation}\label{eq:frequency-problem-main}
\Delta \widehat u_j
+
\bigl( \theta^2-p_j(x)(i\theta)^\alpha-q(x) \bigr)\widehat u_j=0
\quad\text{in }\Omega,
\qquad
\widehat u_j|_{\partial\Omega}=\widehat g .
\end{equation}
By the uniqueness of the time-domain problem and the Fourier--Laplace transform,
the Dirichlet problem \eqref{eq:frequency-problem-main} has a unique solution
for almost every \(\theta=\sigma-i\gamma\), \(\sigma\in\mathbb R\). Therefore,
for such \(\theta\), we define the corresponding stationary DtN map by
\[
\Lambda_j^\theta \widehat g
:=
\partial_\nu \widehat u_j|_{\partial\Omega}.
\]
Since the trace and the normal trace commute with the Fourier--Laplace transform,
we have
\begin{equation}\label{eq:DN-transform-main}
\widehat{\Lambda_j g}(\cdot,\theta)
=
\Lambda_j^\theta \widehat g(\cdot,\theta),
\qquad
\text{for a.e. }\sigma\in\mathbb R .
\end{equation}

We will also need a frequency-domain estimate for the operator norm. We now recall the fractional integral operator \(J^\alpha\) defined by
\[
(J^\alpha h)(t) = \frac{1}{\Gamma(\alpha)} \int_0^t (t-\tau)^{\alpha-1} h(\tau)\,d\tau .
\]
Its Fourier--Laplace transform satisfies
\[
\widehat{J^\alpha h}(\theta) = (i\theta)^{-\alpha} \widehat h(\theta).
\]
To estimate its weighted norm, write \(k(t) = \frac{1}{\Gamma(\alpha)} t^{\alpha-1}\) and
\[
(J^\alpha h)(t) = (k * h)(t) = \int_0^t k(t-\tau) h(\tau)\,d\tau .
\]
Set \(m(t) = e^{-\gamma t} h(t)\), then \(h(t) = e^{\gamma t} m(t)\) and
\[
e^{-\gamma t} (J^\alpha h)(t) = \int_0^t e^{-\gamma(t-\tau)} k(t-\tau) m(\tau)\,d\tau
= (k_\gamma * m)(t),
\]
where \(k_\gamma(t) = e^{-\gamma t} k(t) = \frac{1}{\Gamma(\alpha)} e^{-\gamma t} t^{\alpha-1}\).
Applying Young's inequality for convolutions,
\[
\|k_\gamma * m\|_{L^2(\mathbb R_+)} \le \|k_\gamma\|_{L^1(\mathbb R_+)} \|m\|_{L^2(\mathbb R_+)} .
\]
A direct computation gives
\[
\|k_\gamma\|_{L^1(\mathbb R_+)} = \frac{1}{\Gamma(\alpha)} \int_0^\infty e^{-\gamma t} t^{\alpha-1}\,dt = \gamma^{-\alpha},
\]
while \(\|m\|_{L^2(\mathbb R_+)} = \|h\|_{\gamma,0}\). Hence
\[
\|J^\alpha h\|_{\gamma,0} \le \gamma^{-\alpha} \|h\|_{\gamma,0}.
\]

Combining the Plancherel identity \eqref{eq:weighted-plancherel-main}, the commutation relation \eqref{eq:DN-transform-main} and the above estimate for \(J^\alpha\) (applied to boundary functions), we obtain for every \(g\in\mathcal H_\gamma^1(\partial\Omega\times\mathbb R_+)\)
\[
\int_{\mathbb R} \bigl\| (i\theta)^{-\alpha} (\Lambda_1^\theta-\Lambda_2^\theta) \widehat g(\cdot,\theta) \bigr\|_{L^2(\partial\Omega)}^2\,d\sigma
= \langle J^\alpha(\Lambda_1-\Lambda_2)g \rangle_{\gamma,0}^2 \le \gamma^{-2\alpha} \langle (\Lambda_1-\Lambda_2)g \rangle_{\gamma,0}^2 .
\]
Since 
\[\langle (\Lambda_1 - \Lambda_2) g \rangle_{\gamma,0} \le \langle g \rangle_{\gamma,1}\,  \| \Lambda_1 - \Lambda_2 \|_{\mathcal B
\left(\mathcal H_\gamma^1(\partial\Omega\times\mathbb R_+),
H_\gamma^0(\partial\Omega\times\mathbb R_+)\right)}  = \varepsilon \, \langle g \rangle_{\gamma,1},\] 
where \(\varepsilon := \|\Lambda_1-\Lambda_2\|_{\mathcal B(\mathcal H_\gamma^1(\partial\Omega\times\mathbb R_+), H_\gamma^0(\partial\Omega\times\mathbb R_+))}\), then we have
\begin{equation}\label{eq:frequency-DN-estimate-main}
\int_{\mathbb{R}} \| (i\theta)^{-\alpha} (\Lambda_1^\theta - \Lambda_2^\theta) \widehat{g} \|_{L^2(\partial\Omega)}^2 d\sigma \le  \gamma^{-2\alpha} \varepsilon^2 \langle g\rangle_{\gamma,1}^2
\end{equation}

\medskip
\noindent
\textbf{Step 2.  The frequency-domain integral identity.}

Fix \(\theta=\sigma-i\gamma\) and set
\[
V_j(x,\theta) := -p_j(x)(i\theta)^\alpha - q(x).
\]
Then the frequency-domain equation becomes
\[
\Delta u_j + (\theta^2+V_j)u_j = 0 \quad\text{in }\Omega .
\]
Let \(u_j\) be two sufficiently regular solutions of these equations and write \(f_j := u_j|_{\partial\Omega}\). By Green's identity,
\[
\int_\Omega (u_2 \Delta u_1 - u_1 \Delta u_2)\,dx
= \int_{\partial\Omega} \Bigl( u_2 \frac{\partial u_1}{\partial\nu} - u_1 \frac{\partial u_2}{\partial\nu} \Bigr) dS .
\]
Set \(f_j := u_j|_{\partial\Omega}\) and recall that \(\partial_\nu u_j|_{\partial\Omega} = \Lambda_j^\theta f_j\). Then the right-hand side becomes
\[
\int_{\partial\Omega} \bigl( f_2 \Lambda_1^\theta f_1 - f_1 \Lambda_2^\theta f_2 \bigr) dS .
\]
On the left-hand side, using the equations \(\Delta u_j = -(\theta^2+V_j)u_j\), we obtain
\[
u_2\Delta u_1 - u_1\Delta u_2 = (V_2-V_1)u_1u_2 .
\]
Thus,
\[
\int_\Omega (V_2-V_1) u_1 u_2\,dx = \int_{\partial\Omega} \bigl( f_2 \Lambda_1^\theta f_1 - f_1 \Lambda_2^\theta f_2 \bigr) dS .
\]
Because the DtN map is symmetric with respect to the bilinear pairing without complex conjugation,
\[
\int_{\partial\Omega} f_2 \Lambda_1^\theta f_1\,dS = \int_{\partial\Omega} f_1 \Lambda_1^\theta f_2\,dS .
\]
Consequently,
\begin{equation}\label{eq:alessandrini-main}
\int_\Omega (V_2-V_1) u_1 u_2\,dx = \int_{\partial\Omega} f_1 (\Lambda_1^\theta-\Lambda_2^\theta) f_2\,dS .
\end{equation}
Noticing that
\[
V_2-V_1 = -p_2(i\theta)^\alpha - q + p_1(i\theta)^\alpha + q = P(i\theta)^\alpha,\qquad P:=p_1-p_2,
\]
identity \eqref{eq:alessandrini-main} turns into
\begin{equation}\label{eq:basic-identity-main}
(i\theta)^\alpha \int_\Omega P(x) u_1(x) u_2(x)\,dx
= \int_{\partial\Omega} f_1 (\Lambda_1^\theta-\Lambda_2^\theta) f_2\,dS .
\end{equation}

\medskip
\noindent
\textbf{Step 3.  Construction of complex geometrical optics solutions.}

We construct CGO solutions for
\[
\Delta u_j + (\theta^2+V_j)u_j = 0 \quad\text{in }\Omega,\qquad j=1,2.
\]
By assumption, \(p_j,q\in C^2(\overline\Omega)\) and \(\|p_j\|_{C^2(\overline\Omega)}+\|q\|_{C^2(\overline\Omega)}\le K\). Hence
\[
V_j(\cdot,\theta)\in C^2(\overline\Omega),\qquad
\|V_j(\cdot,\theta)\|_{C^2(\overline\Omega)} \le C K (1+|\theta|^\alpha).
\]

Apply Lemma \ref{lem ell eq} with \(\ell=2\) and \(a(x)=V_j(x,\theta)\). There exist constants \(C_1,C_2,C_3,C_4>0\), depending only on \(\Omega\) and \(\ell=2\), such that whenever \(\zeta_j\in\mathbb C^n\) satisfies
\[
\zeta_j\cdot\zeta_j + \theta^2 = 0,\quad |\operatorname{Im}\zeta_j|\ge\sqrt2,\quad |\zeta_j|\ge C_1\|V_j(\cdot,\theta)\|_{C^2(\overline\Omega)},
\]
the equation admits a solution of the form
\[
u_j(x,\zeta_j,\theta) = e^{\zeta_j\cdot x}\bigl( 1 + w_j(x,\zeta_j,\theta) \bigr),
\]
with the estimate
\begin{equation}\label{eq:cgo-remainder-main}
\|w_j(\cdot,\zeta_j,\theta)\|_{H^2(\Omega)} \le \frac{C_2}{|\zeta_j|} \|V_j(\cdot,\theta)\|_{C^2(\overline\Omega)} .
\end{equation}
Inserting the bound for \(V_j\) yields
\begin{equation}\label{eq:cgo-remainder-main-2}
\|w_j\|_{H^2(\Omega)} \le C \frac{1+|\theta|^\alpha}{|\zeta_j|}.
\end{equation}
Moreover, Lemma \ref{lem ell eq} provides the norm estimate
\begin{equation}\label{eq:cgo-trace-main}
\|u_j(\cdot,\zeta_j,\theta)\|_{H^2(\Omega)} \le C e^{C|\zeta_j|}.
\end{equation}

Now we choose concrete vectors \(\zeta_j\). Take \(\gamma\ge\gamma_0\) sufficiently large so that for all \(\sigma\in\mathbb R\),
\[
C_1 C K (1+|\theta|^\alpha) \le |\theta|,
\]
which is possible since \(|\theta|\ge\gamma\) and \(0<\alpha<1\). Fix \(\xi\in\mathbb R^n\), let \(\rho>0\), take \(\sigma\in[0,\rho]\) and \(r\ge\sqrt2\). Since \(n\ge3\), we can pick real vectors \(\mu,\lambda\in\mathbb R^n\) satisfying
\[
\xi\cdot\mu = 0,\quad \xi\cdot\lambda = 0,
\]
and
\[
|\mu|^2 = \sigma^2 + r^2,\qquad
|\lambda|^2 = \frac{|\xi|^2}{4} + \gamma^2 + r^2,\qquad
\mu\cdot\lambda = \sigma\gamma .
\]
The last condition is fulfilled by adjusting the angle between \(\mu\) and \(\lambda\) in the plane orthogonal to \(\xi\), because \(|\mu||\lambda|\ge\sigma\gamma\).

Define
\[
\zeta_j = -\frac{i\xi}{2} + (-1)^j (i\mu+\lambda),\qquad j=1,2,
\]
so that \(\zeta_1+\zeta_2 = -i\xi\). A direct computation gives
\[
\zeta_j\cdot\zeta_j = -\frac{|\xi|^2}{4} - |\mu|^2 + |\lambda|^2 + 2i\mu\cdot\lambda
= \gamma^2 - \sigma^2 + 2i\sigma\gamma = -\theta^2,
\]
hence \(\zeta_j\cdot\zeta_j+\theta^2=0\). Moreover,
\[
|\operatorname{Im}\zeta_j| = \Bigl|(-1)^j\mu - \frac{\xi}{2}\Bigr|
= \Bigl(|\mu|^2 + \frac{|\xi|^2}{4}\Bigr)^{1/2} \ge r \ge \sqrt2,
\]
and
\[
|\zeta_j|^2 = \frac{|\xi|^2}{2} + |\theta|^2 + 2r^2 \ge |\theta|^2 .
\]
Thus \(|\zeta_j|\ge |\theta| \ge C_1\|V_j\|_{C^2(\overline\Omega)}\). All the assumptions of Lemma \ref{lem ell eq} are thus satisfied, and the CGO solutions constructed above are well defined.

\medskip
\noindent
\textbf{Step 4.  Low-frequency decomposition.}

Insert the CGO ansatz into the identity \eqref{eq:basic-identity-main}. Using
\[
u_1u_2 = e^{(\zeta_1+\zeta_2)\cdot x} (1+w_1)(1+w_2) = e^{-i\xi\cdot x} (1 + w_1 + w_2 + w_1w_2),
\]
we obtain
\[
(i\theta)^\alpha \biggl[ \int_\Omega P(x) e^{-i\xi\cdot x}\,dx
+ \int_\Omega P(x) e^{-i\xi\cdot x} (w_1+w_2+w_1w_2)\,dx \biggr]
= \int_{\partial\Omega} f_1 (\Lambda_1^\theta-\Lambda_2^\theta) f_2\,dS .
\]
Define the Fourier transform of the zero extension of \(P\) by
\[
\widetilde P(\xi) := (2\pi)^{-n/2} \int_\Omega P(x) e^{-i\xi\cdot x}\,dx,
\]
and the remainder
\[
E(\xi,\theta) := \int_\Omega P(x) e^{-i\xi\cdot x} (w_1+w_2+w_1w_2)\,dx .
\]
The above relation rewrites as
\[
(i\theta)^\alpha \bigl[ (2\pi)^{n/2} \widetilde P(\xi) + E(\xi,\theta) \bigr]
= \int_{\partial\Omega} f_1 (\Lambda_1^\theta-\Lambda_2^\theta) f_2\,dS .
\]
Equivalently,
\begin{equation}\label{eq:P-decomposition-main}
\widetilde P(\xi) = B(\xi,\theta) + \mathcal E(\xi,\theta),
\end{equation}
where
\[
B(\xi,\theta) := \frac{(i\theta)^{-\alpha}}{(2\pi)^{n/2}} \int_{\partial\Omega} f_1 (\Lambda_1^\theta-\Lambda_2^\theta) f_2\,dS,\qquad
\mathcal E(\xi,\theta) := -\frac{1}{(2\pi)^{n/2}} E(\xi,\theta).
\]
Consequently,
\begin{equation}\label{eq:P-square-main}
|\widetilde P(\xi)|^2 \le C\bigl( |B(\xi,\theta)|^2 + |\mathcal E(\xi,\theta)|^2 \bigr).
\end{equation}

\medskip
\noindent
\textbf{Step 5.  Estimates of the remainder and the boundary term.}

We first estimate \(\mathcal E(\xi,\theta)\). From \eqref{eq:cgo-remainder-main-2} and \(|\zeta_j|\ge r\),
\[
\|w_j\|_{L^2(\Omega)} \le C\frac{1+|\theta|^\alpha}{r}.
\]
Using \(\|P\|_{L^\infty(\Omega)}\le 2K\) and Hölder's inequality,
\[
|E(\xi,\theta)| \le \|P\|_{L^\infty(\Omega)} \int_\Omega (|w_1|+|w_2|+|w_1w_2|)\,dx
\le C \Bigl( \frac{1+|\theta|^\alpha}{r} + \frac{(1+|\theta|^\alpha)^2}{r^2} \Bigr).
\]
For \(r\ge1\) and large \(|\theta|\) the dominating term is \(\frac{(1+|\theta|^\alpha)^2}{r}\); hence
\begin{equation}\label{eq:E-bound-main}
|\mathcal E(\xi,\theta)| \le C \frac{(1+|\theta|^\alpha)^2}{r}.
\end{equation}

Now we turn to the boundary term \(B(\xi,\theta)\). By the trace theorem and \eqref{eq:cgo-trace-main},
\[
\|f_1\|_{L^2(\partial\Omega)} \le C\|u_1\|_{H^1(\Omega)} \le C e^{C|\zeta_1|},
\]
so that
\begin{equation}\label{eq:B-pre-main}
|B(\xi,\theta)| \le C e^{C|\zeta_1|} \bigl\| (i\theta)^{-\alpha} (\Lambda_1^\theta-\Lambda_2^\theta) f_2 \bigr\|_{L^2(\partial\Omega)} .
\end{equation}
To exploit \eqref{eq:frequency-DN-estimate-main}, we construct a normalized boundary function. For fixed \(\xi\in\mathbb R^n\) and \(r\ge\sqrt2\), define
\[
\widehat g_{\xi,r}(\cdot,\sigma) := \frac{f_2(\cdot,\sigma)}{e^{C|\zeta_2|}(1+\sigma^2)}, \quad \sigma\in \mathbb{R},
\]
where \(C>0\) is chosen large enough so that, by the trace theorem and the CGO estimates,
\[
\|\widehat g_{\xi,r}(\cdot,\sigma)\|_{H^1(\partial\Omega)} \le \frac{C}{1+\sigma^2}.
\]
Then
\[
\int_{\mathbb R} (1+\sigma^2) \|\widehat g_{\xi,r}(\cdot,\sigma)\|_{H^1(\partial\Omega)}^2 d\sigma \le C.
\]
Therefore, the inverse Fourier--Laplace transform \(g_{\xi,r}\) belongs to
\(\mathcal H_\gamma^1(\partial\Omega\times\mathbb R_+)\), and \(\langle g_{\xi,r}\rangle_{\gamma,1}\le C\). Applying \eqref{eq:frequency-DN-estimate-main} to \(g_{\xi,r}\) gives
\[
\int_0^\rho \bigl\| (i\theta)^{-\alpha} (\Lambda_1^\theta-\Lambda_2^\theta) \widehat g_{\xi,r}(\cdot,\sigma) \bigr\|_{L^2(\partial\Omega)}^2 d\sigma \le C\varepsilon^2 .
\]
Since \(f_2 = e^{C|\zeta_2|}(1+\sigma^2)\widehat g_{\xi,r}\) and \(|\zeta_2|\le C\sqrt{|\xi|^2+\rho^2+r^2}\) for \(\sigma\in[0,\rho]\), we obtain
\[
\int_0^\rho \bigl\| (i\theta)^{-\alpha} (\Lambda_1^\theta-\Lambda_2^\theta) f_2 \bigr\|_{L^2(\partial\Omega)}^2 d\sigma
\le C e^{C\sqrt{|\xi|^2+\rho^2+r^2}} (1+\rho^2)^2 \varepsilon^2 .
\]
Inserting this into \eqref{eq:B-pre-main} and integrating in \(\sigma\) yields
\begin{equation}\label{eq:B-sigma-main}
\int_0^\rho |B(\xi,\sigma-i\gamma)|^2 d\sigma \le C e^{C\sqrt{|\xi|^2+\rho^2+r^2}} (1+\rho^2)^2 \varepsilon^2 .
\end{equation}

\medskip
\noindent
\textbf{Step 6.  Low-frequency accumulation and high-frequency cut-off.}

Consider the region
\[
D_\rho := \{ (\xi,\sigma)\in\mathbb R^n\times\mathbb R : |\xi|\le\rho,\; 0\le\sigma\le\rho \}.
\]
Since \(\widetilde P(\xi)\) is independent of \(\sigma\), integrating \eqref{eq:P-square-main} over \(D_\rho\) gives
\[
\rho \int_{|\xi|\le\rho} |\widetilde P(\xi)|^2 d\xi
\le C \iint_{D_\rho} \bigl( |B|^2 + |\mathcal E|^2 \bigr) d\sigma d\xi .
\]
Using \eqref{eq:B-sigma-main},
\[
\int_{|\xi|\le\rho}\int_0^\rho |B|^2 d\sigma d\xi
\le C e^{C\sqrt{2\rho^2+r^2}} (1+\rho^2)^2 \rho^n \varepsilon^2 .
\]
From \eqref{eq:E-bound-main}, for \(\rho\ge\gamma\) and \(0\le\sigma\le\rho\) we have \(|\theta|\le C\rho\), hence
\[
|\mathcal E|^2 \le C \frac{(1+\rho^\alpha)^4}{r^2},
\qquad
\int_{|\xi|\le\rho}\int_0^\rho |\mathcal E|^2 d\sigma d\xi \le C \frac{(1+\rho^\alpha)^4}{r^2} \rho^{n+1}.
\]
Therefore,
\[
\rho \int_{|\xi|\le\rho} |\widetilde P(\xi)|^2 d\xi
\le C \Bigl( e^{C\sqrt{2\rho^2+r^2}} (1+\rho^2)^2 \rho^n \varepsilon^2
+ \frac{(1+\rho^\alpha)^4}{r^2} \rho^{n+1} \Bigr).
\]
Dividing by \(\rho\) we obtain the low-frequency estimate
\begin{equation}\label{eq:low-frequency-main}
\int_{|\xi|\le\rho} |\widetilde P(\xi)|^2 d\xi
\le C \Bigl( e^{C\sqrt{2\rho^2+r^2}} (1+\rho^2)^2 \rho^{n-1} \varepsilon^2
+ \frac{(1+\rho^\alpha)^4}{r^2} \rho^n \Bigr).
\end{equation}

For the high frequencies, fix \(s\in(0,1/2)\). Since \(P=p_1-p_2\in
C^2(\overline\Omega)\), its zero extension to \(\mathbb R^n\), still denoted
by \(P\), belongs to \(H^s(\mathbb R^n)\) and
\[
\|P\|_{H^s(\mathbb R^n)}^2 = \int_{\mathbb R^n} (1+|\xi|^2)^s |\widetilde P(\xi)|^2 d\xi \le C K^2 .
\]
Hence, for \(\rho\ge1\),
\begin{equation}\label{eq:high-frequency-main}
\int_{|\xi|>\rho} |\widetilde P(\xi)|^2 d\xi
\le (1+\rho^2)^{-s} \int_{\mathbb R^n} (1+|\xi|^2)^s |\widetilde P(\xi)|^2 d\xi
\le C \rho^{-2s}.
\end{equation}

Combining \eqref{eq:low-frequency-main} and \eqref{eq:high-frequency-main}, we get
\begin{equation}\label{eq:global-main}
\int_{\mathbb R^n} |\widetilde P(\xi)|^2 d\xi
\le C \Bigl( e^{C\sqrt{2\rho^2+r^2}} (1+\rho^2)^2 \rho^{n-1} \varepsilon^2
+ \frac{(1+\rho^\alpha)^4}{r^2} \rho^n + \rho^{-2s} \Bigr).
\end{equation}

\medskip
\noindent
\textbf{Step 7.  Logarithmic stability estimate.}

We now choose the free parameters to balance the right-hand side. Let
\[
r = \rho^\delta,\qquad \delta := \frac{n+4\alpha+2s}{2}.
\]
Since \(n\ge3\), \(0<\alpha<1\) and \(s>0\), we have \(\delta>1\). For large \(\rho\),
\[
\sqrt{2\rho^2+r^2} \le C\rho^\delta,\qquad
(1+\rho^2)^2\rho^{n-1} \le C\rho^{n+3},
\]
and
\[
\frac{(1+\rho^\alpha)^4}{r^2} \rho^n \le C \rho^{n+4\alpha-2\delta} = C \rho^{-2s}.
\]
Thus \eqref{eq:global-main} simplifies to
\begin{equation}\label{eq:pre-log-main}
\int_{\mathbb R^n} |\widetilde P(\xi)|^2 d\xi
\le C \Bigl( e^{C\rho^\delta} \rho^{n+3} \varepsilon^2 + \rho^{-2s} \Bigr).
\end{equation}

Now relate \(\rho\) to \(\varepsilon\) by requiring
\[
C_0 \rho^\delta + (n+3+2s)\log\rho = \log\frac1{\varepsilon^2},
\]
where \(C_0>0\) is a sufficiently large constant (larger than the constant appearing in the exponent of \eqref{eq:pre-log-main}). Then \(\rho\to\infty\) as \(\varepsilon\to0\) and \(\rho^\delta \sim \frac{1}{C_0}\log\frac1{\varepsilon^2}\). This choice guarantees
\[
e^{C\rho^\delta} \rho^{n+3} \varepsilon^2 \le C \rho^{-2s},
\]
so that \eqref{eq:pre-log-main} reduces to
\[
\int_{\mathbb R^n} |\widetilde P(\xi)|^2 d\xi \le C \rho^{-2s}.
\]
Because \(\rho^\delta \ge c\log\frac1\varepsilon\) for small \(\varepsilon\), we have \(\rho^{-2s} \le C \bigl(\log\frac1\varepsilon\bigr)^{-2s/\delta}\). With \(\delta\) as above,
\[
\frac{2s}{\delta} = \frac{4s}{n+4\alpha+2s}.
\]
Hence
\begin{equation}\label{eq:fourier-final-main}
\int_{\mathbb R^n} |\widetilde P(\xi)|^2 d\xi
\le C \Bigl(\log\frac1\varepsilon\Bigr)^{-\frac{4s}{n+4\alpha+2s}} .
\end{equation}

Finally, let \(P=p_1-p_2\) be extended by zero to \(\mathbb R^n\). By
Plancherel's theorem and by the definition of \(\widetilde P\), we have
\[
\|p_1-p_2\|_{L^2(\Omega)}^2
=
\|P\|_{L^2(\mathbb R^n)}^2
=
\int_{\mathbb R^n}|\widetilde P(\xi)|^2\,d\xi .
\]
Combining this with \eqref{eq:fourier-final-main} and taking square roots, we obtain
\[
\|p_1-p_2\|_{L^2(\Omega)} \le C \Bigl(\log\frac1\varepsilon\Bigr)^{-\tau},
\qquad \tau = \frac{2s}{n+4\alpha+2s}.
\]
Therefore,
\[
\|p_1-p_2\|_{L^2(\Omega)} \le \omega(\varepsilon),
\]
where \(\omega(\varepsilon) \sim C \bigl(\log (1/\varepsilon)\bigr)^{-\tau}\) as \(\varepsilon\to 0\). This is precisely the conditional logarithmic stability estimate stated in Theorem \ref{thm1.3}.
\end{proof}


\section{Concluding remarks}\label{sec5}

By applying the Fourier--Laplace transform in time, we reduce the time-fractional damped wave equation to a family of elliptic equations in the frequency
domain. We then construct complex geometrical optics solutions for the resulting
elliptic problems and derive an integral identity relating the difference of
two damping coefficients to the difference of their DtN maps.
A decomposition into low- and high-frequency components yields a conditional
logarithmic stability estimate for the recovery of the damping coefficient. The
logarithmic modulus reflects the severe ill-posedness of the inverse problem.
Moreover, the stability exponent depends explicitly on the fractional order
\(\alpha\); for fixed dimension \(n\) and fixed \(s\in(0,1/2)\), the exponent decreases
as \(\alpha\) increases. Hence the stability deteriorates as the order of the
nonlocal damping term becomes larger. In particular, the stability estimate also
implies uniqueness of the damping coefficient from the Dirichlet-to-Neumann map.


\section*{Acknowledgments}

This work was partially supported by the National Natural Science Foundation of China
(No. 12271277) and the Open Research Fund of Key Laboratory of Nonlinear Analysis \&
Applications (Central China Normal University), Ministry of Education, China. The second author also thanks Ningbo Youth Leading Talent Project (No.2024QL045).

\bibliographystyle{unsrt}

\end{document}